\documentclass{ws-jktr}

\usepackage {mathpartir}
\usepackage{bcprules}

\usepackage{graphicx} 
\usepackage{amsfonts}
\usepackage{amstext}
\usepackage{latexsym}
\usepackage{amssymb}
\usepackage{color}

\newcommand{\ldb}{[\![}
\newcommand{\rdb}{]\!]}

\newcommand{\plogp}{\leftarrow}

\newcommand{\newkw}{\nu}

\newcommand{\id}[1]{\texttt{#1}}

\newcommand{\pzero}{\mathbin{0}}

\newcommand{\juxtap}{\mathbin{\id{|}}}

\newcommand{\scong}{\mathbin{\equiv}}

\newcommand{\freenames}[1]{\mathbin{\mathcal{FN}(#1)}}
\newcommand{\boundnames}[1]{\mathbin{\mathcal{BN}(#1)}}

\newcommand{\newp}[2]{\id{(}\newkw \; #1 \id{)} #2}

\newcommand{\meaningof}[1]{\ldb #1 \rdb}

\newcommand{\red}{\rightarrow}
\newcommand{\wred}{\Rightarrow}

\newcommand{\rel}[1]{\;{\mathcal #1}\;} %relation
\newcommand{\bc}{\mathbin{\mathbf{::=}}}
\newlength{\ltext}
\newlength{\lmath}
\newlength{\cmath}
\newlength{\rmath}
\newlength{\rtext}

 \newcommand{\abs}[1]{\left\vert#1\right\vert}
\begin{document}

\title{Knots as processes: a new kind of invariant}
%\titlerunning{Knots as processes}

\author{ L.G. Meredith \and David F. Snyder }
\address{ Partner, Biosimilarity LLC\\ 505 N72nd St, Seattle, WA 98103, USA, \\
  \email{ lgreg.meredith@biosimilarity.com }}
\address{Department of Mathematics\\ Texas State University---San Marcos\\ 601 University Drive \\ San Marcos, TX 78666 \\
  \email{ dsnyder@txstate.edu }}

\maketitle              % typeset the title of the contribution

%%% ----------------------------------------------------------------------

\begin{abstract}

          We exhibit an encoding of knots into processes in the
          $\pi$-calculus such that knots are ambient isotopic if and
          only their encodings are weakly bisimilar.

\end{abstract}

\keywords{Process calculi, knots, invariants}

% \begin{keyword}
% concurrency, message-passing, process calculus, reflection, program logic
% \end{keyword}

%\end{frontmatter}

% section front matter (end)

\section{Introduction}\label{sec:introduction} % (fold)
Knot and link tabulation continues to be a lively area of scientific
research and promises to be useful to areas of science such as quantum
computing and DNA unpacking \cite{Quantum1} \cite{Quantum2} \cite{DNA1}.  The past
20 years have seen major advances in knot classification, the
development of knot invariants, and computational methods for
tabulating knots and links.  There are tables of the alternating prime
knots of up to 23 crossings \cite{Hoste2005The-enumeration}
\cite{Rankin2004Enumerating-I} \cite{Rankin2004Enumerating-II} \cite{KnotsTo16}. While
current algorithms provide complete tables of knots (and links),
winnowing these tables of duplicates is a time consuming
task. Moreover, as knot tables have proved of use to researchers in
genetics, and may prove to be to researchers in quantum computation,
the need to search these tables in meaningful ways presents
itself. For example, of the 4,976,016,485 prime, non-oriented,
alternating knots with minimal crossing number of 22, which contain
the tangle corresponding to 5/3 (if any)?  Of course, knot invariants
are useful to distinguish knots, but few provide the basis for a formal
language of knot properties with which to identify classes of knots \emph{via}  logical
expressions in the language. 

\subsection{Summary of contributions and outline of paper}

Here we present a newly-found strong knot invariant that does give
rise to such a formal language of knot properties. Knots are
invariantly associated to expressions in Milner's $\pi$-calculus
\cite{milner91polyadicpi}, a member of a family of dynamic calculi
known as the mobile process calculi. These calculi were developed for
the analysis of concurrent computation \cite{ParrowWalker}
\cite{SangiorgiWalker}.  The invariant introduced here interprets knot equivalence up
to ambient isotopy as an equivalence of process dynamics known as
bisimulation \cite{DBLP:conf/lics/Sangiorgi04} \cite{Sangiorgi95} in
the concurrency theory literature. Of critical importance, and
somewhat surprisingly, these two notions of equivalence correspond
exactly on the image of the encoding: the main result of this paper is
that two knots are ambient isotopic if, and only if, their images
under the encoding are weakly bisimilar.

%\begin{theorem*}[main]
%There is a process calculus $\mathcal{P}$ and a procedure for encoding knots into $\mathcal{P}$ such that the encoding is well-defined and injective on the collection of ambient isotopy classes of knots.  
%\end{theorem}

Building on this result, we observe that dual to the process calculi
are a family of logics, the Hennessy-Milner logics (HMLs), providing a
logical language capable of the classification of processes via logical
properties. Factoring this capability through our encoding leads to the
development of a logical language identifying classes of knots with
logical properties. In particular, the spatial logics discovered
by Caires and Cardelli \cite{CairesC03} constitute
a particularly interesting sub-family of the HMLs having logical
connectives enabling us to take full advantage of key features of our
encoding. We illustrate the application of the logic in \cite{Caires04}
primarily \emph{via} examples of predicates that select features of
knots.

The paper is organized as follows. First, we provide a brief overview
of knot presentations relevant to this paper, given in the context of
the history of knot tabulation. Next, we give a brief overview of
process calculi, highlighting achievements in that field pertinent to
this paper. Section \ref{sec:knots_as_processes} begins the technical
exegesis with an introduction to Milner's polyadic $\pi$-calculus via
demonstration of process expressions that, much like the combinators
of Conway's tangle calculus, reflect aspects of knot structure
(crossings \emph{etc.}). The section concludes with an example
encoding of the trefoil knot. Section
\ref{sub:the_syntax_and_semantics_of_the_notation_system} gives a
condensed but complete formal presentation of the process
calculus. The main theorem is proved in section
\ref{sub:main_thm_proof_sketch}. This is followed by an account of
spatial logic via knots in section \ref{sub:characteristic_formulae}
and how this can be applied in the study and inquiry of knot tables,
with potential application in knot tabulation algorithms
themselves. Finally, new techniques are judged not solely by the
questions they answer but by the new questions to which they give
rise. In the conclusion, we identify areas of further investigation.
% section introduction (end) 
 
% section introduction (end)

\section{A summary of knot presentations}\label{sec:notation} % (fold)

The knot presentations we focus upon herein are the Gauss code \cite{Read1977On-the-Gauss-cr}, the
Dowker-Thistle\-thwaite (DT) code \cite{DT}, the Conway notation \cite{Conway1970An-enumeration-}, and the Master
Code of Rankin, Schermann and Smith \cite{Rankin2004Enumerating-I} \cite{Rankin2004Enumerating-II}. Other presentation schemes, such
as braid representatives and Morse links, are not directly useful to
the development of this paper and thus are overlooked. We also give a
``wish-list'' of properties one would desire from one's knot notation
scheme.

\subsection{Knot presentations}
This section briefly discusses various knot notations and establishes
some terminology, with references for readers needing further details.

\subsubsection{Regular knot projections}

In 1847, J.B. Listing classified knots up to 5 crossings by analyzing
knot projections \cite{Listing1848Vorstudien-zur-}. Listing was the
first to publish an article containing a drawing of a regular knot
projection (or \emph{knot diagram} or \emph{planar diagram}). A
\emph{knot} is taken to be a polygonal simple closed curve in
$3$-space. Choose any plane on which projection onto the plane results
in a curve whose only singularities are transverse double points which
are not the image of any vertex from the polygonal curve. Using a
canonical normal vector to the plane for line of sight, for a given
double point $d$ we find the point $c$ in the projection's preimage of
$d$ farthest from our line of sight, and we indicate an undercrossing
by erasing the image of a small neighborhood of $c$ from the planar
projection. A careful presentation is given in \cite{LivingstonText}.
	
Out of all possible regular projections of a knot, there is (at least)
one with minimal number of crossings. 

% We'll write $\#(K)$ to denote this number.

\subsubsection{The Gauss code.} Gauss used regular knot projections to
arrive at what is now called the Gauss Code of a knot
\cite{Read1977On-the-Gauss-cr}. The Gauss code is the first knot notation
system. P. G. Tait used an encoding in the 1870's to classify knots up
to 7 crossings \cite{Tait}. Tait's encoding may be regarded as an
extension to the Gauss Code. One begins somewhere on the knot
projection (not at a crossing point), then proceeds along the knot
applying labels to the first, third, fifth etc. crossing until all
crossings are labeled; one then traverses the knot once more, writing
the label of each crossing in the order that you reach it, attaching a
plus or minus sign, depending on whether you are crossing over or
under.

\subsubsection{Tait, Little, and Kirkman.} Using Kirkman's
classification of certain polygons \cite{Kirkman1885The-enumeration},
Tait (and Little, using similar methods) was able to tabulate knots up
to 11 crossings \cite{Tait} \cite{Little1885On-knots-with-a}
\cite{Little1890Alternate-pm-kn}
\cite{Little1900Non-alternate-p}. Tait's system included using a
simple reduction rewrite strategy within his notation system. This notation was improved by Dowker and Thistlethwaite, as discussed in some detail below.

\subsubsection{Reidemeister moves.} Reidemeister developed a reduction
rewrite system for knot projections (“the three Reidemeister moves”)
\cite{ReidemeisterMoves}. We illustrate the three moves. The
convention in these (and in other graphically-based invariants) is
that we show only the relevant part of the knot. In our diagrams here,
we have placed dashed circles, which the reader should imagine as
viewing through scope lens only a portion of the knot, with
the remainder of the knot outside of the scope of view and unchanged.

\begin{figure}[hbt]
    \centering
    \scalebox{0.37}[0.370]{\includegraphics[viewport=0 0 390 360]{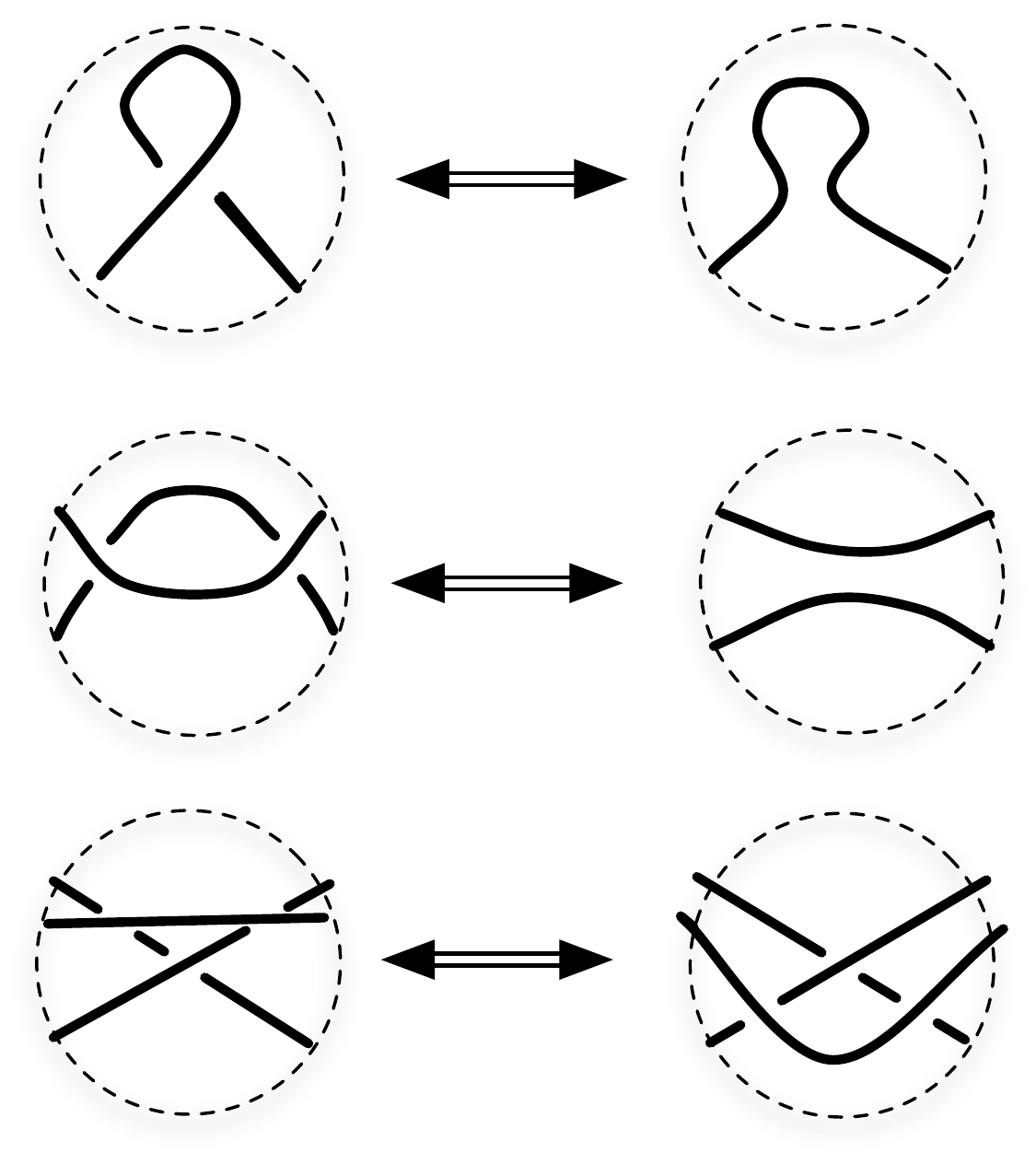}}
    \caption{ The three Reidemeister moves. The topmost move will be referred to as $R_1$, the middle one as $R_2$, and the bottommost as $R_3$. For any one of these moves $R$, we write $R^{\to}(K)$ (resp. $R^\leftarrow$(K)) if the move has been applied left-to-right (resp. right-to-left)  to the knot diagram $K$.}
\label{fig:RMoves}
\end{figure}

The import of this system is Reidemeister's Theorem: an ambient
isotopy of two knots may be exhibited as sequences of rewrites of (any of) their
diagrams, and vice versa. That is, two knots are ambient isotopic if and only if a
diagram of one may be rewritten, via a sequence of Reidemeister
moves, to a diagram of the other \cite{LivingstonText}.

\subsubsection{Knotation} Conway developed a clever notational system
for tangles, finding an algebraic-like system for this system that led to several methods of reduction rewrites
\cite{Conway1970An-enumeration-}. He used this system to retabulate
knots and links of 11 crossings by hand, in one afternoon (an effort
that took Tait and Little years of work), discovering one omission and
a few duplications in the process. Conway's system can be used to
classify all arithmetic (also called algebraic, or rational)
knots. The Conway code for a knot originally was given as a “basic
polyhedron” followed by a sorted list of arithmetic tangles. See
Conway's paper \cite{Conway1970An-enumeration-} for details. This system has extensions by Caudron
\cite{Caudron1981-Classification} and by Bangert
\cite{Bangert2002Algorithmic-Pro}. Conway has investigated further
structures such as bangles and bracelets.

\subsubsection{Dowker-Thistlethwaite Codes (DT codes)}\label{ssub:dowker_thistlethwaite_codes} % (fold)
Dowker created a variation of Tait's notational system that is easier
to implement computationally. Dowker and Thistlethwaite made it the
basis for an algorithm that successfully enumerated knots of up to 13
crossings \cite{DT}. Not every DT code is valid, i.e. an arbitrary DT
code may not correspond to an actual knot, and two distinct composite
knots may share the same DT Code. However, a valid DT code for a prime
knot specifies the knot uniquely \cite{SchareinPhD}. In their paper,
Dowker and Thistlethwaite develop an algorithm to filter out invalid
cases. They also give a reduction system to remove duplicates from
their enumeration.

We now recall the definition of the DT code of
a regular knot projection \cite{DT}. Begin at a selected point on the
diagram (excepting crossing points) and traverse the diagram in a
selected direction. At each encounter of a crossing, label the
crossing with the next available counting number; for an even label,
prepend a negative sign if traversing the overcross. Returning to the
beginning point, each crossing has been labeled twice, once with an
odd number and once with an even number. Let $S_\pm$ denote the set of
labels generated.

If the knot projection has $n$ crossings, then there are $2n$ labels
(from $1$ to $2n$ in absolute value) and at each crossing an unique
odd counting number paired with an even integer. This induces a
parity-reversing function $\sigma$ on $S= \{1,\ldots,2n\}$, where for
each $i\in S$ we find the crossing with label whose absolute value is $i$ and let
$\sigma(i)$ be the absolute value of the other label at that crossing.
Note that $\sigma(\sigma(i))=i$ for any $i\in S$. Note that there is
also a bijective translation $\tau: S\to S_\pm$ where
$\abs{\tau(s)}=s$ for all $s\in S$.  Then the DT code of the knot
projection is determined by $\delta=\tau\circ\sigma$, usually given as
the list $\{\delta(1),\delta(3),\ldots,\delta(2n-1)\}$. Note that
$\delta\circ\sigma=\tau$. Finally, for each $i\in S$, let $sgn(i)$
denote the sign of $\tau(i)$ (so $sgn(i)<0$ implies $i$ is even).

For each $i\in S_\pm$, let $C(i)$ denote the crossing with label
$i$. Then for any $i\in S$ we know $C(i)$ is adjacent to the two
crossings $C(i\pm 1)$ and $C(\delta(i)\pm 1))$, where addition is modulo $2n$. This fact is applied
later when giving an explicit algorithm for encoding a knot.

% subsubsection dowker_thistlethwaite_codes (end)
 
\subsubsection{From Calvo to Rankin, Flint, and Schermann} Calvo
developed an inductive knot tabulation algorithm, thereby sidestepping the need to
check validity of DT codes \cite{Calvo1997Knot-enumeratio}. However, detection of
duplication within the tables generated then becomes the issue. Calvo had the insight that
understanding the deeper flype structure of prime, non-alternating
diagram led to greater efficiencies in his algorithm. The Calvo
algorithm was essentially refined in the development of a notational
system by Rankin, Flint, and Schermann , based on what they call
the group code which reminds one somewhat of the Gauss code but, instead, using a
Conway-like insertion scheme to allow for easy reduction of flype
structures in the notation \cite{Rankin2004Enumerating-I}
\cite{Rankin2004Enumerating-II} \cite{Rankin2004EnumeratingLinks}. This results in a master array that provides a unique identifier for  prime alternating knots. This forms part of the basis for inductive construction of all prime alternating knots of crossing size $n$ from those of crossing size $n-1$(see the above papers for details).  Composite knots remain problematic, due principally to the issue of detection of duplicates.

\subsection{Desirable properties for a knot presentation system }\label{sub:desirable_properties_for_a_knot_notation_system_} % (fold)

In this section, we list some properties that one would wish for a
knot presentation system to enjoy, followed by a discussion of the
knot presentations of the last subsection in the light of these
wishes.

Due to complexity considerations, one may well despair of a knot
notation system that would allow for mathematical classification of
all knots. However, in designing a system of knot notation or
presentation, informed by the history of knot tabulation and
classification we can enumerate the following properties such a
notation system may enjoy. These properties invariably correspond to
demands of functoriality on the encoding when considered as a map from
the category of knots or braids to some suitable target category while
others are demands on the faithfulness (respectively, fullness) of the
encoding considered as functor.

\begin{itemize}

\item {Surjection (realizability).} Each code in the notation represents
  a knot (or if this is not the case, those codes that do not
  represent a knot are easily recognizable). A code for a knot should
  be easily obtained from one of its planar diagrams.

%\item [Reduction] 
\item {Reduction.} The notation enjoys a calculus with which to reduce
  and simplify encodings. Each step of simplification or reduction
  results in an encoding representing a knot isotopy equivalent to the
  originating knot.

%\item [Minimality] 
\item {Minimality.} The notational encoding of a given knot can be
  reduced to a (finite non-empty set of equivalent) minimal
  encoding(s).

%\item [Injection] 
\item {Well-definedness.} If a knot has two (or more) minimal reduced
  encodings, each of these encodings is equivalent to the notation for
  a knot equivalent \emph{via} ambient isotopy to the original.

\item {Injection.} Two knots that have equivalent
  minimal reduced encodings are ambient isotopic.

	\item {Economy.} The notation is computationally cheap and easily constructed from a diagram of the knot or from a simple code, such as the DT code.
	
%\item [Compositionality] 
\item {Compositionality.} Operations on knots, e.g. knot
  composition, correspond to natural operations on elements in the
  image of the encoding.
  
%\item [Separation] 
\item {Separation.} The notation can be used to classify a class X of
  knots, where X contains a previously classified class of knots (for
  example, arithmetic knots (also known as algebraic knots) and
  bracelets) but is not previously classified itself.

%\item [Classification] 
\item {Extensibility.} The notation enjoys a formal language in which
  to describe properties and invariants of notation objects that
  reflect interesting properties of knots. This language should also
  be useful in selecting specific sets of knots (such as the set of
  all 21-crossing prime alternating links containing the tangle
  5/3). This language ideally should be scaleable and be applicable to
  tables of indefinite size.
\end{itemize}

The DT code, with proper care taken, satisfies the
first six properties but breaks under compositionality. Conway's
tangle notation enjoys all of the properties listed, though the Conway
encoding requires foreknowledge of the classes of basic polygons with
$n$ vertices. The master array of Rankin, Flint, and Schermann 
satisfies the first five properties, but appears unlikely to support
the last four properties to a degree useful for anything other than
tabulation.

\subsubsection{A meta-knotation}

 Due to Reidemeister's Theorem, in order to establish our main theorem we need only consider diagrams of knots and not knots themselves. Hence we  adopt the following convenient
notation and terminology. The meta-variable  $K$
ranges over knot diagrams and we reserve $\mathcal{K}$ to range over knots. For simplicity, we refer to $K$ as a knot. We write $\#(K)$ to mean the number of crossing points in the diagram $K$ but will call this value the crossing number of the knot $K$, which the reader is asked not to confuse with the minimal crossing number of the knot (it is in this sense that we will later say that ``two knots from the same isotopy class can have different crossing numbers'').   Formally speaking $\sim_{\text{iso}}$ will denote  equivalence of knots under ambient
isotopy and $\sim_{R}$ will denote equivalence of diagrams under (finite) sequences of
Reidemeister transforms. However, when there is little chance of confusion we
drop the subscript and write $\sim$.  

% subsection desirable_properties_for_a_knot_notation_system_ (end)

% section notation (end) 

% section notation (end)

\section{Concurrent process calculi and spatial logics }\label{sec:concurrent_process_calculi_and_spatial_logics_} % (fold)
In the last thirty years the process calculi have matured into a
remarkably powerful analytic tool for reasoning about concurrent and
distributed systems. Process-calculus-based algebraical specification of
processes began with Milner's Calculus for Communicating Systems (CCS)
\cite{MilnerCCS80} and Hoare's Communicating Sequential Processes
(CSP) \cite{CSP} \cite{CSP1} \cite{CSP2} \cite{CSP3}, and continue
through the development of the so-called mobile process calculi,
e.g. Milner, Parrow and Walker's $\pi$-calculus \cite{ParrowWalker},
Cardelli and Caires's spatial logic \cite{CairesC04} \cite{CairesC03}
\cite{Caires04}, or Meredith and Radestock's reflective calculi
\cite{MeredithR05} \cite{meredith2005rho}. The process-calculus-based
algebraical specification of processes has expanded its scope of
applicability to include the specification, analysis, simulation and
execution of processes in domains such as:

\begin{itemize}
\item telecommunications, networking, security and application level protocols
\cite{AbadiB02} 
\cite{AbadiB03} 
\cite{BrownLM05} 
\cite{LaneveZ05}; 
\item programming language semantics and design
\cite{BrownLM05}
\cite{djoin}
\cite{JoCaml}
\cite{WojciechowskiS99};
\item webservices
\cite{BrownLM05}
\cite{LaneveZ05}
\cite{MeredithB03};
\item and biological systems
\cite{Cardelli04}
\cite{DanosL03}
\cite{RegevS03}
\cite{PriamiRSS01}.
\end{itemize}

Among the many reasons for the continued success of this approach are
two central points. First, the process algebras provide a
compositional approach to the specification, analysis and execution of
concurrent and distributed systems. Owing to Milner's original
insights into computation as interaction \cite{Milner93}, the process
calculi are so organized that the behavior ---the semantics--- of a
system may be composed from the behavior of its components
\cite{Fokkink}. This means that specifications can be constructed in
terms of components ---without a global view of the system--- and
assembled into increasingly complete descriptions.

The second central point is that process algebras have a potent proof
principle, yielding a wide range of effective and novel proof
techniques \cite{MilnerS92} \cite{SangiorgiWalker} \cite{Sangiorgi95}
\cite{hop}. In particular, \emph{bisimulation} encapsulates an effective
notion of process equivalence that has been used in applications as
far-ranging as algorithmic games semantics
\cite{Abramsky2005Algorithmic-Gam} and the construction of
model-checkers \cite{Caires04}. The essential notion can be stated in
an intuitively recursive formulation: a \emph{bisimulation} between two
processes $P$ and $Q$ is an equivalence relation $E$ relating $P$
and $Q$ such that: whatever action of $P$ can be observed, taking it
to a new state $P'$, can be observed of $Q$, taking it to a new state
$Q'$, such that $P'$ is related to $Q'$ by $E$ and vice versa. $P$ and
$Q$ are \emph{bisimilar} if there is some bisimulation relating
them. Part of what makes this notion so robust and widely applicable
is that it is parameterized in the actions observable of processes
$P$ and $Q$, thus providing a framework for a broad range of
equivalences and up-to techniques \cite{milner92techniques} all governed by the same core
principle \cite{SangiorgiWalker} \cite{Sangiorgi95} \cite{hop}.
% section concurrent_process_calculi_and_spatial_logics_ (end)

% section concurrent_process_calculi_and_spatial_logics_ (end)
    
\section{Knots as processes}\label{sec:knots_as_processes} % (fold)

This section bootstraps intuitions about the target calculus by
introducing process expressions for key aspects of a knot's
structure. An $n$ crossing knot $K$ is modeled as a system
$\meaningof{K}$ of concurrently executing processes. More
specifically, $\meaningof{K}$ is a \emph{parallel composition} $\Pi_{i
  = 0}^{n-1} \meaningof{C(i)} | W$ of $n+1$ processes consisting of
$n$ crossing processes $\meaningof{C(i)}$ and a process $W$
constituting a ``wiring harness.'' The latter process can be thought
of as the computational equivalent to Conway's ``basic polygon,'' if
the knot is in minimal crossing number form. The complete expression
of the encoding is

% \begin{eqnarray*}
%   \meaningof{K} & := &  (v_0 ... v_{4n-1})( \Pi_{i = 0}^{n-1} (\nu \; u)\meaningof{C(i)}(v_{4i},...,v_{4i+3},u) \\
%   & & \; \; \; \; \; \; \; \; \; \; \; \; \; \; \; \; \; \; | \Pi_{i = 0}^{n-1} W(v_{\omega(i,0)},v_{\omega(i,1)})|W(v_{\omega(i,2)},v_{\omega(i,3)}) ) \nonumber
% \end{eqnarray*}

\begin{align*}
  \meaningof{K} :=  (v_0 ... v_{4n-1}) & ( \Pi_{i = 0}^{n-1} (\nu \; u)\meaningof{C(i)}(v_{4i},...,v_{4i+3},u) \\
  & | \Pi_{i = 0}^{n-1} W(v_{\omega(i,0)},v_{\omega(i,1)})|W(v_{\omega(i,2)},v_{\omega(i,3)}) )
\end{align*}

Here, $C(i)$ represents the $i$th crossing in the knot diagram $K$.  The wiring process, $\Pi_{i =
  0}^{n-1}W(v_{\omega(i,0)},v_{\omega(i,1)})|W(v_{\omega(i,2)},v_{\omega(i,3)})$,
is itself a parallel composition of wire processes that correspond to
edges in the 4-valent graph of the knot shadow \cite{SchareinPhD} the constraints of
which are reflected in the indexing function $\omega$. The wiring process
may be calculated from other knot notations. For example,  we later show how the indexing function $\omega$ may be calculated
from $\delta$, the DT code of the knot projection. The crossing and wire
processes have further substructure, outlined below.

\begin{remark}[knots as abstractions]
  The reader familiar with process calculi will observe that the
  encoding actually produces an \emph{abstraction} \cite{SangiorgiWalker}  in $4\#(K)$ names
  (see the first choice in the production labeled agent in the grammar
  of \ref{sub:the_syntax_and_semantics_of_the_notation_system}). This is actually a way of demarking that the encoding should
  be insensitive to the particular set of names chosen to represent the
  ports of the crossings. Some caution must be exercised, however, as
  the encoding only preserves knot structure if the abstraction is
  applied to a vector of $4\#(K)$ distinct names.
\end{remark}

\subsection{Crossing processes}\label{sub:crossing_processes} % (fold)
A crossing is conceived in the diagram below as a black-box having
four points of external interaction (\emph{ports}) with the remainder
of the knot process and as having two internal wires each connecting a
pair of ports. As a process, the crossing has four possible behaviors,
as shown in the defining encoding below.
\begin{figure}[hbp]
    \centering
    \scalebox{0.27}[0.270]{\includegraphics[viewport=0 0 390 360]{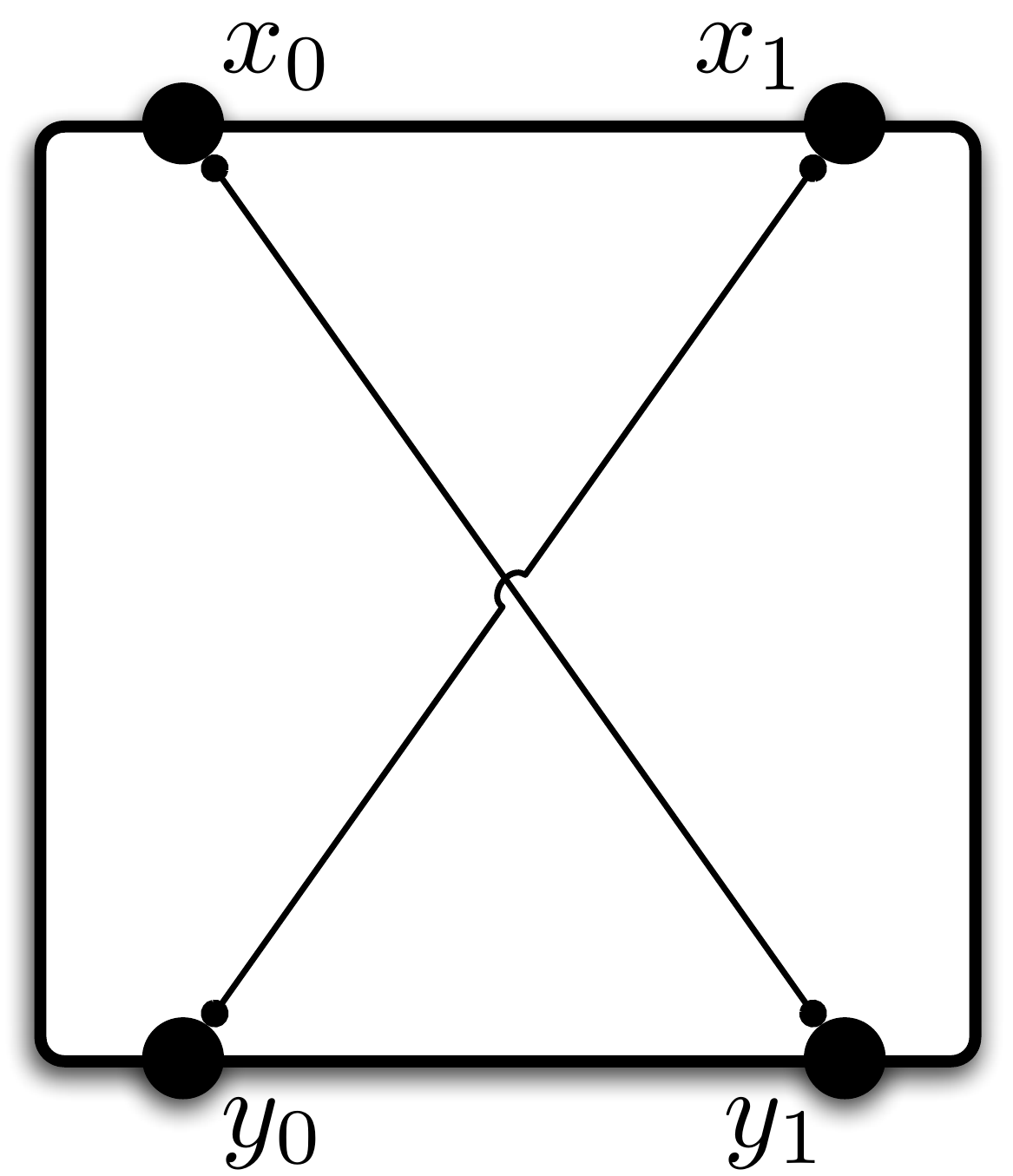}}
    \caption{ Crossing process }
\end{figure}
%\begin{eqnarray*}
%   C(x_0,x_1,y_0,y_1,u) & \plogp   & x_1?(s).y_0!(s).(C(x_0,x_1,y_0,y_1,u)|u!) \nonumber \\
%  & & + y_0?(s).x_1!(s).(C(x_0,x_1,y_0,y_1,u)|u!) \nonumber \\
%  & & + x_0?(s).u?.y_1!(s).(C(x_0,x_1,y_0,y_1,u)) \nonumber \\
%  & & + y_1?(s).u?.x_0!(s).(C(x_0,x_1,y_0,y_1,u)) \nonumber
%\end{eqnarray*}

\begin{eqnarray*}
   C(x_0,x_1,y_0,y_1,u) & \plogp   & x_1?(s).y_0!(s).(C(x_0,x_1,y_0,y_1,u)|u!) \nonumber \\
  & & + y_0?(s).x_1!(s).(C(x_0,x_1,y_0,y_1,u)|u!) \nonumber \\
  & & + x_0?(s).(C'(x_0,x_1,y_0,y_1,u,y_1)) \nonumber \\
  & & + y_1?(s).(C'(x_0,x_1,y_0,y_1,u,x_0)) \nonumber
\end{eqnarray*}

where

\begin{eqnarray*}
   C'(x_0,x_1,y_0,y_1,u,v) & \plogp   & x_1?(s).y_0!(s).(C'(x_0,x_1,y_0,y_1,u,v)|u!) \nonumber \\
  & & + y_0?(s).x_1!(s).(C'(x_0,x_1,y_0,y_1,u,v)|u!) \nonumber \\
  & & + u?.v!(s).(C(x_0, x_1, y_0, y_1, u)) \nonumber
\end{eqnarray*}

A crossing process has four ports $x_0,x_1,y_0,y_1$ and a hidden
synchronizer $u$. Each port has a partner port, linked as shown in the
diagram (note the relationship to Conway's $\pm 1$ tangles
\cite{Conway1970An-enumeration-}). For example, the first behavior (indicated by the first
term of the summand) is that the process listens at port $x_1$ for a
signal $s$ (which will come, if at all, via the wiring
process). Having heard $s$, the signal is passed directly to the port
$y_0$ where the signal is then broadcast via the wiring process. Then
the process alerts the hidden synchronizer $u$ that a signal has been
passed between the ports, while concurrently preparing itself for
further signal processing. The second summand represents a signal
passing along the same strand in the opposite direction. The third and
fourth summands are similar to the first two, except that before
passing any received signal to its partner port, the process waits for
a signal from the synchronizer $u$ before allowing the signal to
pass, while still allowing a signal to pass between $x_1$ and $y_0$ if the synchronizer has not yet been alerted. So the role of $u$ is that of a traffic controller who gives
priority to traffic over the route between $x_1$ and $y_0$, mimicking
an over-crossing.

\subsection{Wirings}\label{sub:wirings} % (fold)

As an illustration of the expressive power of the formalism, taken
together with the short description of the process calculus in the next
section, the definitions below fully equip the interested reader to
verify that wire processes are lossless, infinite capacity buffers.

\begin{eqnarray*}
    W(x,y) & \plogp & (\nu \; n \; m)(Waiting(x,n,m) | Waiting(y,m,n)) \nonumber \\
Waiting(x,c,n) & \plogp   & x?(v).(\nu \; m)(Cell(n,v,m) | Waiting(x,c,m)) \nonumber \\
  & & + c?(w).c?(c).Ready(x,c,n,w) \nonumber \\
  Ready(x,c,n,w) & \plogp  & x?(v).(\nu \; m)(Cell(n,v,m) | Ready(x,c,m,w)) \nonumber \\
  & & + x!(w).Waiting(x,c,n) \nonumber \\
  Cell(c,v,n) & \plogp & c!(v).c!(n).0 \nonumber
\end{eqnarray*}

The ports $x$ and $y$ in which $W$ is \emph{parameterized} may be
intuitively considered splice points in the knot diagram. We adopt
this terminology in the sequel.

One may well wonder why perfect buffers are chosen to represent
wires. For example, the following process intuitively captures a notion of wire.

\begin{equation*}
  Relay(x,y) := x?(s).y!(s).Relay(x,y) + y?(s).x!(s).Relay(x,y)
\end{equation*}

The problem is one of composition. Foreshadowing the method of proof,
in the sequel we will need to compose wires and crossings and have the
result act as a wire. For example, if $W(x,y)$ represents a candidate
for wire behavior, to model the first Reidemeister move we have a demand that

\begin{equation*}
  W(y_{0},y_{1}) \simeq (\nu \; x_{0} \; x_{1} \; w_{0} \; w_{1} )(W(y_{0},w_{0}) |(\nu \; u)C(x_{0},x_{1},w_{0},w_{1},u) | W(x_{0},x_{1}) | W(y_{1},w_{1}))
\end{equation*}

Again, the reader may verify that while this is true of buffers, it is
not true of the $Relay$ process defined above.

\subsection{An example: the trefoil knot as a process}
The above intuitions are illustrated with an example, showing one way to
represent the trefoil knot as a process. Following the schema above,
the process encoding of the trefoil knot is a parallel composition of
three crossing circuits with a wiring harness whose design ensures that the
crossing circuits are connected to each other in a way that respects
the knot diagram. Additionally, we make each synchronization channel
local to each crossing via a restriction on that channel.

\begin{figure}[tbp]
\begin{center}
{\includegraphics[width=3in]{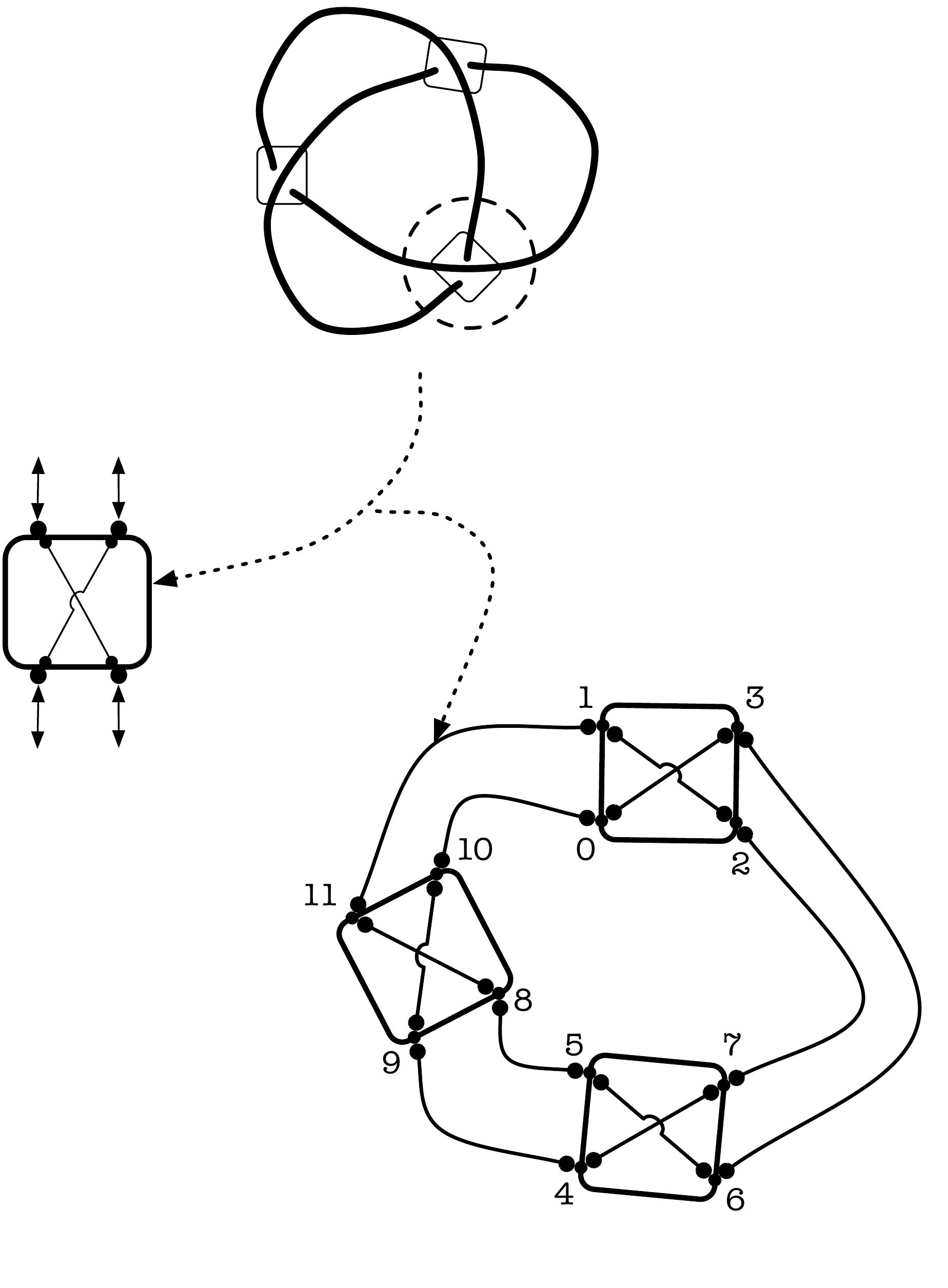}}
\caption{ The $3_1$ (trefoil) knot as process. The ports in the
  circuit diagram have been labeled with the corresponding subscripted
  index in the process expression of the text.}
\end{center}
\end{figure}

% \begin{mathpar}
%   \meaningof{3_1} =
%    (v_0 ... v_{5}) (\nu \; u_0)C(v_0,v_1,v_2,v_3,u_0)
%     | W(v_2,v_7) | W(v_3,v_6)
%     \and \\
%     | (\nu \; u_1)C(v_4,v_5,v_6,v_7,u_1)
%     | W(v_4,v_9) | W(v_5,v_8)
%     \and \\
%     | (\nu \; u_2)C(v_8,v_9,v_{10},v_{11},u_2)
%     | W(v_{10},v_0) | W(v_{11},v_1)
% \end{mathpar}

\begin{align*}\label{eq:trefoil_encoding}
  \meaningof{3_1} = (v_0 ... v_{5}) & (\nu \; u_0) C(v_0,v_1,v_2,v_3,u_0) \\
   & | W(v_2,v_7) | W(v_3,v_6) \\
   & | (\nu \; u_1)C(v_4,v_5,v_6,v_7,u_1) \\
   & | W(v_4,v_9) | W(v_5,v_8) \\
   & | (\nu \; u_2)C(v_8,v_9,v_{10},v_{11},u_2) \\
   & | W(v_{10},v_0) | W(v_{11},v_1) 
\end{align*} 

\subsection{The distinguishing power of dynamics}\label{sub:dynamic distinction} % (fold)

In summary, to each knot $\mathcal{K}$ the encoding associates an
invariant $\meaningof{K}$, an expression in a calculus of
message-passing processes via an encoding of a diagram of the
knot. More precisely, given the set of knot diagrams $\mathbb{K}$ and
the set of processes modulo structural equivalence $\mathbb {P}$ (see section
\ref{sub:the_syntax_and_semantics_of_the_notation_system}), the encoding induces a map, $\meaningof{-}: \mathbb{K} \to \mathbb
{P}$. Most importantly, the notion of equivalence of knots coincides perfectly with the notion of equivalence of
processes, i.e. bisimulation (written here and in the sequel
$\simeq$). Stated more formally,

\begin{theorem}[main]
\begin{eqnarray*}
    K_1 \sim K_2 & \iff & \meaningof{K_1} \simeq \meaningof{K_2}. \nonumber
\end{eqnarray*}
\end{theorem}

%Thus $\mathbb{K}$ may be taken to be the collection of isotopy classes of knots. 
In particular, in contrast to other invariants, the alignment of process
dynamics with knot characteristics is what enables the invariant
identified here to be perfectly distinguishing. As discussed below,
among the other beneficial consequences of this alignment is the
ability to apply process logics, especially the spatial sub-family of
the Hennessy-Milner logics, to reason about knot characteristics and
knot classes.

We see moreover the possibility of a deeper connection. As mentioned
in the previous section, bisimulation has turned out to a powerful
proof technique in the theory of computation adaptable to a wide range
of situations and admitting a number of potent up-to techniques
\cite{DBLP:conf/lics/Sangiorgi04}. One of the central aims of this
research is to broaden the domain of applicability of
bisimulation-based proof methods.

% subsection basic_interpretation (end) 

% subsection basic_interpretation (end)

\subsection{The syntax and semantics of the notation system}\label{sub:the_syntax_and_semantics_of_the_notation_system} % (fold)

% [Ed. Note: stylistic alternative 1:]
% Having established an intuitive account of the target calculus via the
% encoding, we summarize its technical presentation below. These calculi
% are typically presented by giving a set of generators and relations on
% them. The grammar, below, describing term constructors, freely
% generates the set of processes $\mathcal{P}$ under the operations
% indicated.

% [Ed. Note: stylistic alternative 2:]
Having bootstrapped an intuitive account of the target calculus  \emph{via}
the encoding, we now summarize a technical presentation of this calculus. The
typical presentation of such a calculus follows the style of giving
generators and relations on them. The grammar, below, describing term
constructors, freely generates the set of processes,
$\mathcal{P}$. This set is then quotiented by a relation known as
structural congruence, yielding the set $\mathbb {P}$ mentioned above
in the type signature of the map constituting the encoding.

\subsubsection{Process grammar}\label{subsub:process_grammar}

\begin{mathpar}
  \inferrule* [lab=summation] {} {{M,N} \bc \pzero \;|\; x.A \;|\; M+N}
  \and
  \inferrule* [lab=agent] {} {{A} \bc (\vec{x})P \;| \; [\vec{x}]P}
  \and \\
  \inferrule* [lab=process] {} {{P,Q} \bc N \;| \;P|Q \;| X\langle \vec{y} \rangle \;| \; (\textsf{rec} \; X(\vec{x}).P)\langle \vec{y} \rangle \;| \; (\nu \; \vec{x})P}
\end{mathpar} 

% [Ed. Note: part of stylistic alternative 1:] A relation on
% $\mathcal{P}$ known as structural congruence yields the quotient set
% $\mathbb {P}$, our target calculus.
% [Ed. Note: calling $\mathbb {P}$ the target calculus is slightly inaccurate...]

Note that $\vec{x}$ denotes a vector of names of length
$|\vec{x}|$. In the encodings for knots, crossings and wires given above we
adopted the following standard abbreviations.

\begin{mathpar}
   x?(\vec{y}).P := x.(\vec{y})P \and  x!(\vec{y}).P := x.[\vec{y}]P
   \and
   X(\vec{y}) \plogp P := (\vec{y})(\textsf{rec} \; X(\vec{x}).P)\langle \vec{y} \rangle
   \and \Pi_{i=0}^{n-1}P_i := P_0 | \ldots | P_{n-1}
\end{mathpar}

\subsubsection{Structural congruence}

\paragraph{Free and bound names and alpha-equivalence.} At the
core of structural equivalence is alpha-equivalence which identifies
process that are the same up to a change of variable. Formally, we
recognize the distinction between free and bound names. The free names
of a process, $\freenames{P}$, may be calculated recursively as
follows:

\begin{mathpar}
  \freenames{\pzero} := \emptyset
  \and \\
  \freenames{x?(\vec{y}).P} := \{ x \} \cup (\freenames{P} \setminus \{ \vec{y} \})
  \and 
  \freenames{x!(\vec{y}).P} := \{ x \} \cup \{ \vec{y} \} \cup \freenames{P} 
  \and \\
  \freenames{P|Q} := \freenames{P} \cup \freenames{Q}
  \and
  \freenames{P + Q} := \freenames{P} \cup \freenames{Q}
  \and \\
  \freenames{(\nu \; \vec{y})P} := \freenames{P} \setminus \{ \vec{y} \}
  \and  \\
  \freenames{(\textsf{rec} \; X(\vec{x}).P)\langle \vec{y} \rangle}
  := \{ \vec{y} \} \cup \freenames{P} \setminus \{ \vec{x} \}
\end{mathpar}

The bound names of a process, $\boundnames{P}$, are those names occurring in $P$
that are not free. For example, in $x?(y).0$, the name $x$ is free, while $y$ is bound.

\begin{definition}
Then two processes, $P,Q$, are alpha-equivalent if $P = Q\{\vec{y}/\vec{x}\}$ for
some $\vec{x} \in \boundnames{Q},\vec{y} \in \boundnames{P}$, where $Q\{\vec{y}/\vec{x}\}$
denotes the capture-avoiding substitution of $\vec{y}$ for $\vec{x}$ in $Q$.
\end{definition}

\begin{definition}
  The {\em structural congruence} \cite{SangiorgiWalker} , $\equiv$, between processes is the
  least congruence containing alpha-equivalence, satisfying the
  abelian monoid laws (associativity, commutativity and $\pzero$ as
  identity) for parallel composition $|$ and for summation $+$, in addition to the
  following axioms:
\begin{mathpar}
 (\nu \; x)\pzero \equiv \pzero 
 \and
 (\nu \; x)(\nu \; x)P \equiv (\nu \; x)P \and (\nu \; x)(\nu \; y)P \equiv (\nu \; y)(\nu \; x)P 
 \and \\
 P | (\nu \; x)Q \equiv (\nu \; x)(P|Q), \; \mbox{\textit{if} }x \not\in \freenames{P}
 \and
 (\textsf{rec} \; X(\vec{x}).P)\langle \vec{y} \rangle \equiv P\{\vec{y}/\vec{x}\}\{(\textsf{rec} \; X(\vec{x}).P)/X\}
\end{mathpar}
\end{definition}

\subsubsection{Operational semantics} 

Finally, we introduce the computational dynamics. What marks these
algebras as distinct from other more traditionally studied algebraic
structures, e.g. vector spaces or polynomial rings, is the manner in
which dynamics is captured. In traditional structures, dynamics is typically
expressed through morphisms between such structures, as in linear maps
between vector spaces or morphisms between rings. In algebras
associated with the semantics of computation, the dynamics is
expressed as part of the algebraic structure itself, through a
reduction reduction relation typically denoted by $\red$. Below, we
give a recursive presentation of this relation for the calculus used
in the encoding.

\begin{mathpar}
  \inferrule* [lab=Comm] {\vec{y} \cap \vec{v} = \emptyset \\ |\vec{y}| = |\vec{z}|} { x.(\vec{y})P \juxtap x.(\nu \vec{v})[\vec{z}]P \red (\nu \vec{v})(P\{\vec{z}/\vec{y}\} | Q) }
  \and \\
  \inferrule* [lab=Par] {{P} \red {P}'} {{{P} | {Q}} \red {{P}' | {Q}}}
  \and
  \inferrule* [lab=Equiv]{{{P} \scong {P}'} \andalso {{P}' \red {Q}'} \andalso {{Q}' \scong {Q}}}{{P} \red {Q}}
  \and
  \inferrule* [lab=New] {{P} \red {P}'} {{\newp{{x}}{{P}}} \red {\newp{{x}}{{P}'}}}  
\end{mathpar}

We write $\wred$ for $\red^*$, and $P\red$ if $\exists Q $ such that $ P \red Q$.

In closing this summary, we take the opportunity to observe that it is
precisely the dynamics that distinguishes this encoding. The
equivalence that coincides with knot equivalence is a \emph{behavioral}
equivalence, i.e. an equivalence of the dynamics of processes in the
image of the encoding. In a marked departure from Gauss codes or
DT codes or Conway's ``knotation,'' this facet of the encoding affords
the \emph{conflation} of notation scheme with invariant, providing a
framework in which to establish the distinguishing power of the
invariant and a language in which to express classes of knots as
logical properties, as discussed in subsection \ref{sub:characteristic_formulae}.

\subsubsection{Bisimulation}

The computational dynamics gives rise to another kind of equivalence,
the equivalence of computational behavior. As previously mentioned
this is typically captured \emph{via} some form of bisimulation.

% The notion we use in this paper is weak barbed bisimulation
% \cite{milner91polyadicpi}.

The notion we use in this paper is derived from weak barbed
bisimulation \cite{milner91polyadicpi}. We must introduce an ``up to''
\cite{DBLP:conf/concur/SangiorgiM92} \cite{DBLP:conf/concur/Pous06}
strategy to deal with the fact that Reidemeister moves can not only
introduce or eliminate crossings (see $R_1$ $R_2$), but ``reorder''
them (see $R_3$).

\begin{definition}
  An agent $B$ occurs \emph{unguarded} in $A$ if it has an occurence
  in $A$ not guarded by a prefix $x$. A process $P$ is observable at
  $x$, written here $P \downarrow x$, if some agent $x.A$ occurs
  unguarded in $P$. We write $P \Downarrow x$ if there is $Q$ such
  that $P \wred Q$ and $Q \downarrow x$.
\end{definition}

\begin{definition}
%\label{def.bbisim}
A \emph{barbed bisimulation} is a symmetric binary relation 
${\mathcal S}$ between agents such that $P\rel{S}Q$ implies:
\begin{enumerate}
\item If $P \red P'$ then $Q \wred Q'$ and $P'\rel{S} Q'$, for some $Q'$.
\item If $P\downarrow x$, then $Q\Downarrow x$.
\end{enumerate}
$P$ is barbed bisimilar to $Q$, written
$P \simeq Q$, if $P \rel{S} Q$ for a barbed bisimulation ${\mathcal S}$.
\end{definition}

% One of the principal advantages of this framework is the co-algebraic
% proof method for proving bisimilarity between two processes:
% exhibit a bisimulation \cite{DBLP:conf/lics/Sangiorgi04}.

\subsubsection{Contexts}

One of the principle advantages of computational calculi like the
$\pi$-calculus is a well-defined notion of context,
contextual-equivalence and a correlation between
contextual-equivalence and notions of bisimulation. The notion of
context allows the decomposition of a process into (sub-)process and
its syntactic environment, its context. Thus, a context may be
thought of as a process with a ``hole'' (written $\Box$) in it. The
application of a context $M$ to a process $P$, written $M[P]$, is
tantamount to filling the hole in $M$ with $P$. In this paper we do
not need the full weight of this theory, but do make use of the notion
of context in the proof the main theorem. As will be seen, the
Reidmeister moves amount to decomposing the representation of the knot
into some collection of crossings or wires and the rest of the
knot.

\begin{mathpar}
  \inferrule* [lab=summation] {} {{M_{M},M_{N}} \bc \Box \;|\; x.M_{A} \;|\; M_{M}+M_{N}}
  \and
  \inferrule* [lab=agent] {} {{M_{A}} \bc (\vec{x})M_{P} \;| \; [\vec{x}]M_{P}}
  \and \\
  \inferrule* [lab=process] {} {{M_{P}} \bc M_{N} \;| \;P|M_{P} \;| (\textsf{rec} \; X(\vec{x}).M_{P})\langle \vec{y} \rangle \;| \; (\nu \; \vec{x})M_{P}}
\end{mathpar} 

\begin{definition}[contextual application] Given a context $M$, and
  process $P$, we define the \emph{contextual application}, $M[P] :=
  M\{P/\Box\}$. That is, the contextual application of M to P is the
  substitution of $P$ for $\Box$ in $M$.
\end{definition}

\begin{example}\label{example:trefoilcontext} For example, if we take
%   \begin{mathpar}
%   M_{3_1} :=
%    (v_0 ... v_{5}) (\nu \; u_0)C(v_0,v_1,v_2,v_3,u_0)
%     | W(v_2,v_7) | \Box
%     \and \\
%     | (\nu \; u_1)C(v_4,v_5,v_6,v_7,u_1)
%     | W(v_4,v_9) | W(v_5,v_8)
%     \and \\
%     | (\nu \; u_2)C(v_8,v_9,v_{10},v_{11},u_2)
%     | W(v_{10},v_0) | W(v_{11},v_1)
%   \end{mathpar}
  \begin{align*}
    M_{3_1} := (v_0 ... v_{5}) & (\nu \; u_0)C(v_0,v_1,v_2,v_3,u_0) \\
    & | W(v_2,v_7) | \Box \\
    & | (\nu \; u_1)C(v_4,v_5,v_6,v_7,u_1) \\
    & | W(v_4,v_9) | W(v_5,v_8) \\
    & | (\nu \; u_2)C(v_8,v_9,v_{10},v_{11},u_2) \\
    & | W(v_{10},v_0) | W(v_{11},v_1)
  \end{align*}
then $\meaningof{3_1} = M_{3_1}[W(v_3,v_6)]$
\end{example}

% subsection the_syntax_and_semantics_of_the_notation_system (end)    

% section concurrent_process_calculi (end)

\section{Main theorem: proof sketch}\label{sub:main_thm_proof_sketch} % (fold)

We have a couple of technicalities to dispatch. To motivate the
first of these we wish to note that the arguments for the
forward direction require some care. The aim is to capture the
intuitive equivalence-preserving nature of the Reidemeister moves as
corresponding bisimularity-preserving transformations on processes (in
the image of the encoding). Because of the encoding of crossing
information of wires in terms of synchronization of signal-flow, we
have to introduce ``enough signal'' to keep the knot ``firing'', as it
were, to establish that the process transformations corresponding
to the knot transformations are bisimulation-preserving. Rather than
seeing this extra condition as a weakness of the approach we submit that this feature
provides evidence to our claim that the characterization of ambient isotopy of knots at
work in this encoding is in terms of process dynamics.

We will say that the encoding of a knot is \emph{alive} as long as it
is firing, i.e. the process is enabled to make a reduction step. If it
ever ceases to push signal through, i.e. process cannot make a
reduction step, then it is \emph{dead}.
  
We can ascertain an upperbound on initial signal that guarantees
liveness of the encoding. Surely, the parallel composition of $2\#(K)$
barbs, i.e. two barbs for each crossing, will guarantee the liveness
of the encoding. More declaratively, we simply demand that $
\meaningof{K} | initSignal$ be live before we are willing to admit
it as a representation of the knot.

\begin{definition}
  More precisely, we will call the pair
  $(\meaningof{K},initSignal)$ \emph{alive} if $initSignal$ is
  an abstraction over a parallel composition of barbs, with
  $|initSignal| = |\meaningof{K}|$, we demand that for any vector
  of distinct names, $v$ with $|v| = |\meaningof{K}|$ and any state,
  $K'$, such that $\meaningof{K}\langle v \rangle | initSignal
  \langle v \rangle \wred K'$ we have that $K' \red$.
\end{definition}

\paragraph*{Different crossing numbers mean different numbers of free
  names.}
Knots in the same isotopy class may have diagrams with different numbers of
crossings. Different numbers of crossings lead to different arities
in the abstractions, so in interpreting these knots we haveto work to properly capture the notion of equivalence. While the precise statement
is somewhat technical, the intuition is simple: it is possible to find
a common ``interface'', i.e. argument list, between the two knots such
that restricting to that argument list obtains bisimilar processes.

Imagine (the processes that interpret) knots as programs housed in
boxes with ports on the perimeter. Two knot diagrams $K_i, i \in \{1,2\}$, from
the same isotopy class may have different crossing numbers and thus
their boxes have different numbers $\#(K_i)$ of ports on the outside. We can find a number of ports, call it $n$, somewhere
between the minimal crossing number of the isotopy class and the
lesser of $\#(K_i)$ such that if -- using restriction -- we close off
$\#(K_1) - n$ ports on the first box and $\#(K_2) - n$ on the second
we get two boxes that perform the same observable set of signal
processing steps.

Formally, suppose $K_{1} \sim K_{2}$ and let $\#_{Min}(K) :=
min\{\#(K') : K' \sim K \}$. We assert that there is an $n$ such that   
$4\#_{Min}(K_1) \leq n \leq 4*min\{ \#(K_1), \#(K_2) \}$ and for
any vector of names, $\vec{v}$, with $|\vec{v}| = n$ and $v[i] \neq
v[j] \iff i \neq j $, there exists two vectors of names, $\vec{w_1},
\vec{w_2}$, also all distinct, such that
\begin{eqnarray}
    (\nu \; \vec{w_1})\meaningof{K_{1}}\langle \vec{v}:\vec{w_1} \rangle & \simeq & (\nu \; \vec{w_2})\meaningof{K_{2}}\langle \vec{v}:\vec{w_2} \rangle \nonumber
  \end{eqnarray}
  with $|\vec{w_i}| = 4\#(K_{i}) - n$.

  \paragraph{Prime versus composite knots.} Finally, we need to say a
  little about how we obtain a $\pi$-calculus expression for a given
  knot diagram. We use another bootstrapping procedure, beginning with another
  knot notation (Dowker-Thistlethwaite codes is used here) and
  exhibiting an algorithm for calculating a process expression from the chosen notation scheme. The reader will
  note that DT codes are  unique only when restricted to the class of prime knots. It turns out
  that our encoding preserves knot composition. In fact, knot
  composition turns out to be a specialized form of a procedure,
  parallel composition + hiding, long-investigated in the
  process-algebraic setting. So, it is sufficient to demonstrate the
  encoding for prime knots.

\subsection{Forward direction. $ K_1 \sim K_2  \implies  \meaningof{K_1} \simeq \meaningof{K_2}$}

\paragraph*{Strategy and intuitions.} Since $K_1 \sim K_2$ we know
there is a sequence of Reidemeister moves converting $K_1$ to 
$K_2$. Each move is proved to correspond to a bisimilarity-preserving
transformation on a related process. We establish context and
substitution lemmas and as a consequence obtain that the process operations
corresponding to the Reidemeister moves preserve bisimularity. As
noted above, a small amount of bookkeeping is required to iteratively
apply these transformations to mirror Reidemeister moves as applied in a
proof of ambient isotopy of two knots.

We begin by observing that the intuition behind the Reidmeister moves
is fundamentally about performing local operations, i.e. on
some subset of wires or crossings, while leaving the rest of the
process unchanged. We interpret this notion of a ``local operation'',
say $R$, on a knot diagram $K$ schematically as follows.

\begin{itemize}
\item factor the process, $\meaningof{K}$, as $M[\Pi_iC_i | \Pi_jW_j]$
  where $\Pi_iC_i | \Pi_iW_i$ encodes the set of crossings or wires
  to be modified by $R$, i.e. the left side of the move to be
  performed, and $M$ is the context representing the unchanged portion of the process;
\item letting $R^{\to}(K)$ (resp. $R^{\leftarrow}(K)$) denote the
  left-to-right (resp. right-to-left) application of the move R to
  $K$, then $\meaningof{R^{\to}(K)}$ is calculated as
  $M[\Pi_{i'}C_{i'}' | \Pi_{j'}W_{j'}]$, where $\Pi_{i'}C_{i'}' |
  \Pi_{j'}W_{j'}$ are the processes interpreting the modified set of
  crossings or wires, i.e. the right side of the move to be performed.
\end{itemize}

The mathematical content of these statements is that the encoding
naturally extends to an encoding $\meaningof{R_{i}\{L,R\}}$ of the left and right hand sides of each
Reidemeister move such that $\meaningof{K} =
M[\meaningof{R_{i}L}]$ (resp. $M[\meaningof{R_{i}R}]$) and
$\meaningof{R_{i}^{\to}(K)} = M[\meaningof{R_{i}R}]$
(resp. $\meaningof{R_{i}^{\leftarrow}(K)} = M[\meaningof{R_{i}L}]$). Here, a picture really is worth a thousand words (see figure
\ref{fig:RMovesAsXforms}).  

\begin{example} For example,  as in
  \ref{example:trefoilcontext} taking $M_{3_1}$ to be  the encoding of $R_1(3_1,W(v_3,v_6))$,
  the application of $R_1$ to $3_1$ at to the strand corresponding to
  the wire process $W(v_3,v_6)$ is given by
  $\meaningof{R_1(3_1,W(v_3,v_6))} = (\nu \; x_{0} \; x_{1} \; y_{0}
  \; y_{1})M_{3_1}[((\nu \; u)C(x_{0},x_{1},y_{0},y_{1},u) |
  W(x_{0},x_{1}) | W(y_{0},v_{3}) | W(y_{1},v_{6}))]$. See figure
  \ref{fig:TrefoilContextIllustration}.
  
  \begin{figure}[tbp]
    \centering
    \scalebox{0.30}[0.300]{\includegraphics[viewport=30 30 810 550]{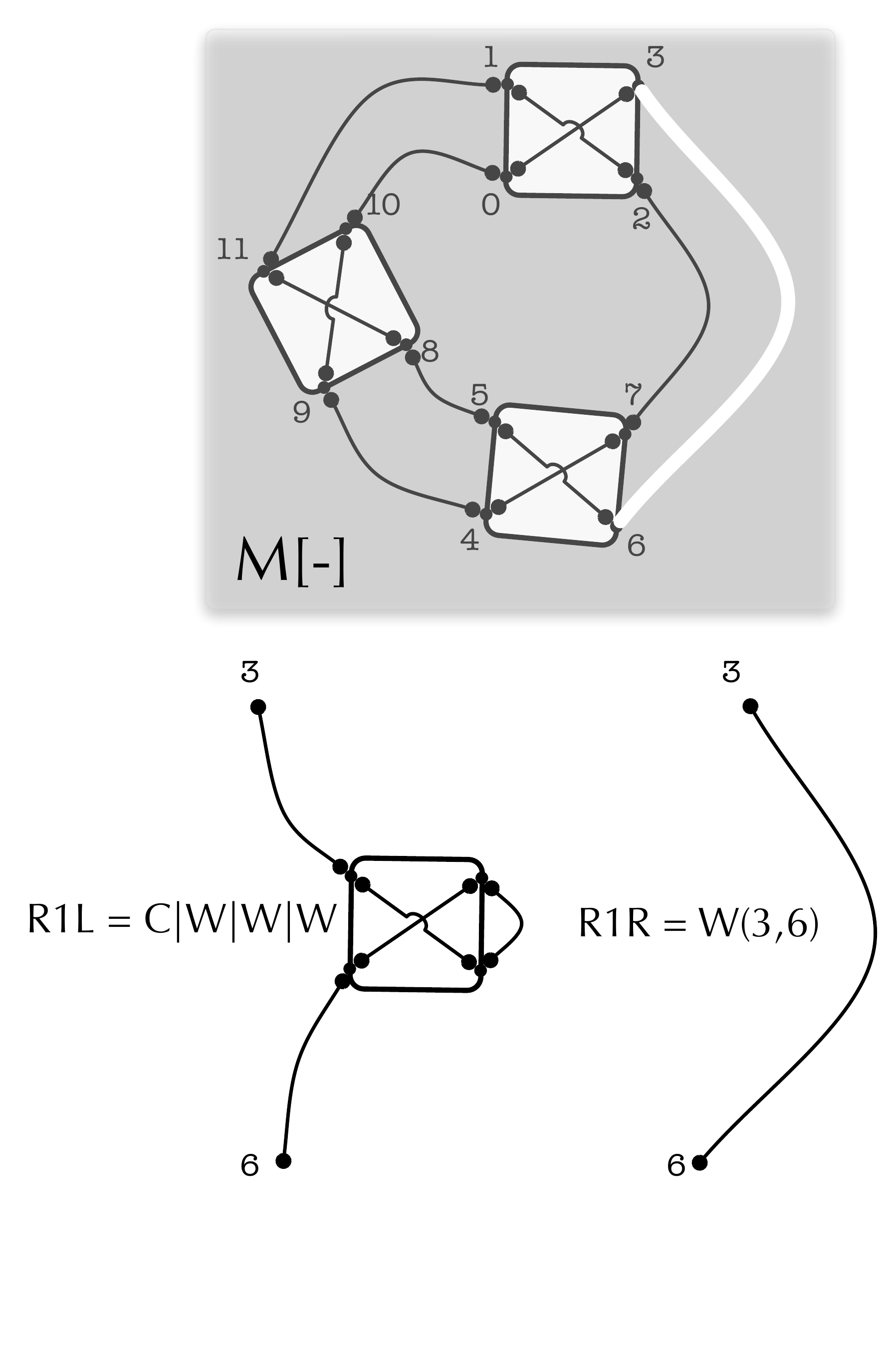}}
    \caption{The figure illustrates a Reidemeister move of the first type as reflected in the context $M$ of the $3_1$  knot, as given in Ex. \ref{example:trefoilcontext}. }
  \label{fig:TrefoilContextIllustration}
\end{figure}
  % The interface of the $R_1$ move is ...
\end{example}

A certain discipline is required in the extended encoding. To
establish our substitution and context lemmas we have to keep the
\textit{interface} (\emph{i.e.} the ports in the process expression
corresponding to the splice points) of the left and right hand sides
of the R-move the same. So, for $R_{1}L$ and $R_{2}L$ we must restrict the ports
that are not the splice points. One way to address this is to embed
the restrictions into the encodings of $R_{1}L$ and $R_{2}L$. Algebraically,

  % \begin{eqnarray}
%     \lefteqn{\meaningof{R_{1}L}(y0,y1) =} \nonumber \\
%     & & (\nu \; x_{0} \; x_{1} \; v_{0} \; v_{1} ) ((\nu \; u)C(x_{0},x_{1},v_{0},v_{1},u) | W(x_{0},x_{1}) | W(y_{0},v_{0}) | W(y_{0},v_{1})) \nonumber \\
%     \lefteqn{\meaningof{R_{2}L}(x_{00},x_{01},x_{10},x_{11}) =} \nonumber \\
%     & & (\nu \; y_{00},y_{01},y_{10},y_{11},)((\nu \; u_{0})C(x_{00},x_{01},y_{00},y_{01},u{0}) \nonumber \\
%     & & | W(y_{00},y_{11}) | W(y_{01},y_{10}) | (\nu \; u_{1})C(x_{10},x_{11},y_{10},y_{11},u_1)) \nonumber
%   \end{eqnarray}

\begin{align*}
  \meaningof{R_{1}L}(y0,y1) = & \\
  (\nu \; x_{0} \; x_{1} \; w_{0} \; w_{1} ) & (W(y_{0},w_{0}) \\
  & |(\nu \; u)C(x_{0},x_{1},w_{0},w_{1},u) | W(x_{0},x_{1}) \\
  & | W(y_{1},w_{1})) \\
  \meaningof{R_{2}L}(x_{00},x_{01},x_{10},x_{11}) = & \\
  (\nu \; y_{00},y_{01},y_{10},y_{11},w_{00},w_{01},w_{10},w_{11}) & (W(x_{00},w_{00}) | W(x_{01},w_{01}) \\
  & | (\nu \; u_{0})C(w_{00},w_{01},y_{00},y_{01},u{0}) \\
  & | W(y_{00},y_{11}) | W(y_{01},y_{10}) \\
  & | (\nu \; u_{1})C(w_{10},w_{11},y_{10},y_{11},u_1) \\
  & | W(x_{10},w_{10}) | W(x_{11},w_{11}))
\end{align*}

  Keeping to the notion of equivalence outlined in the section above,
  however, it will be convenient to factor out the restrictions as in
  the example above.

\begin{figure}[tbp]
%  \centering
 % \scalebox{0.35}[0.350]{\includegraphics[viewport=30 30 810 550]{ReidemeisterMovesAsCircuits071115.pdf}}
% Reidemeistermovesascircuits21072006

\center\includegraphics[width=3in]{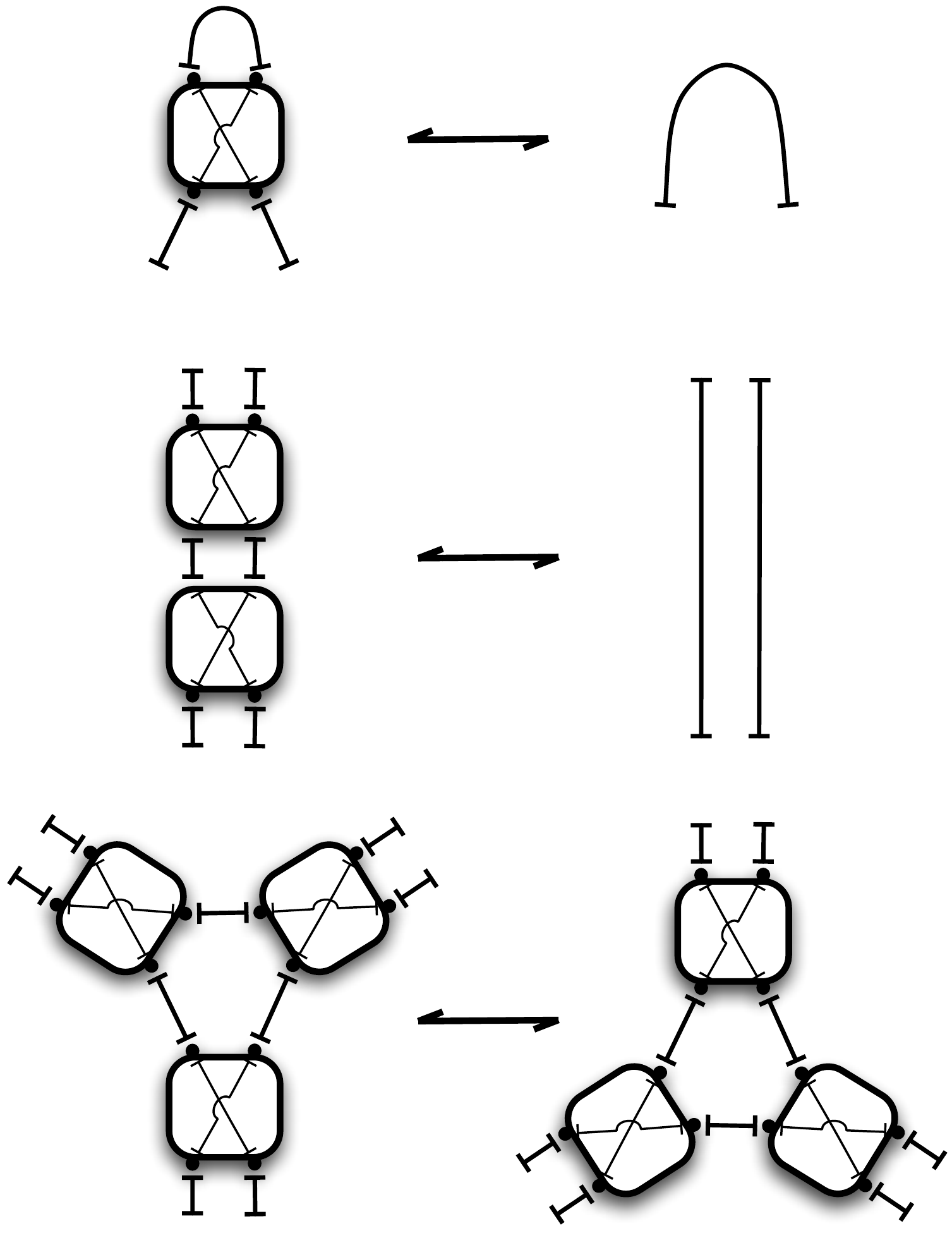}  
\caption{ Reidemeister moves as bisimilarity-preserving transformations. Cf. Fig. \ref{fig:RMoves}}

\label{fig:RMovesAsXforms}
\end{figure}

\begin{lemma}[context]
$\forall i \in \{ 1, 2, 3 \}$ if $K_{1} \stackrel{R_{i}}{\rightarrow}
K_{2}$ then there exists a context $M$ and (possibly empty) vector of
distinct names $\vec{w}$ such that
  \begin{eqnarray}
    (\nu \; \vec{w})\meaningof{K_{1}}\langle v:w\rangle & = & (\nu \; \vec{w})M[ \meaningof{R_{i}L} ] \nonumber \\
    \meaningof{K_{2}} & = & M[ \meaningof{R_{i}R} ] \nonumber
  \end{eqnarray}
\label{context}
\end{lemma}

\begin{proof}
  This follows directly from the definition of the encoding.
\end{proof}

\begin{lemma}[substitution]
  For all $i \in \{ 1, 2, 3 \}$ $R_{i}L$ is bisimilar to $R_{i}R$ in the
  context of a live encoding. That is, if
  \begin{itemize}
    \item $\meaningof{K} | initSignal$ is alive, and
    \item $\meaningof{K} | initSignal = M[ \meaningof{R_{i}L} ]$
  \end{itemize}
for some context $M$ then we can substitute $\meaningof{R_{i}R}$ in
  its place without change of behavior, i.e.
  \begin{eqnarray}
    (\nu \; \vec{w})M[ \meaningof{R_{i}L} ] & \simeq & M[ \meaningof{R_{i}R} ] \; ,\ \forall i \in \{ 1, 2, 3 \} \nonumber
  \end{eqnarray}
\end{lemma}

\begin{proof}
  This follows from the definitions of $\meaningof{R_{i}\{L,R\}}$ plus
  the requirement that $M$ derives from a live encoding. Even if one
  side has additional synchronizations, there is always enough signal
  to overcome spurious blocking.
\end{proof}

An immediate consequence of these two lemmas is 

\begin{lemma}[1-step]
  if $K_{1}$ can be derived from $K_{2}$ by application of one Reidemeister move, then 

  \begin{eqnarray}
    \meaningof{K_{1}}\langle v \rangle & \simeq & (\nu \; w)\meaningof{K_{2}}\langle v:w \rangle \nonumber
  \end{eqnarray}
\end{lemma}

We are almost in a position to complete the proof of the forward
implication. The next step is to show that you can iterate the
performance of these bisimilarity-preserving transformations in a way
that precisely mimics iterated application of Reidemeister
moves. Because the moves $R_{1}$ and $R_{2}$ change the number of
crossings the number of ports in the interface of the (encoding of
the) knot to which they are applied changes. This is the core issue in
the formal statement of the theorem. Because the splice interface
remains constant across the two sides of a Reidemeister move all one
has to do is keep track of the ports to be hidden. The complete proof
uses a case analysis of the composition of any two types of
Reidemeister moves. We illustrate the analysis in the case of a
simplifying (crossing-elimination) $R_1$ step followed by a complicating
(crossing-introducing) $R_2$ step.

\begin{itemize}
   \item The $R_{1}L$ $\to$ $R_{1}R$ step means we have a context $M$ such that
     \begin{eqnarray}
       (\nu \; x_0 \; x_1 \; w_0 \; w_1)\lefteqn{\meaningof{K_{1}}\langle \vec{v_0}:x_0 : x_1 : w_0 : w_1\rangle} \nonumber \\
       & & = (\nu \; x_0 \; x_1 \; w_0 \; w_1)M[ \meaningof{R_{1}L} ] \nonumber \\
       & & \simeq M[ \meaningof{R_{1}R} ] \nonumber \\
       & & = \meaningof{K_{2}}\langle \vec{v_0} \rangle. \nonumber
     \end{eqnarray}
   \end{itemize}

   \begin{itemize}
   \item The  $R_{2}R$ $\to$ $R_{2}L$ step means we have a context $M'$ such that
     \begin{eqnarray}
       (\nu \; y_{00} \ldots y_{11} \; w_{00} \ldots w_{11})\lefteqn{\meaningof{K_{3}}\langle \vec{v_1}:y_{00}:\ldots : y_{11} : w_{00} : \ldots : w_{11} \rangle} \nonumber \\
       & & = (\nu \; y_{00} \ldots y_{11} \; w_{00} \ldots w_{11})M'[ \meaningof{R_{2}L} ] \nonumber \\
       & & \simeq M'[ \meaningof{R_{1}R} ] \nonumber \\
       & & = \meaningof{K_{2}}\langle \vec{v_1} \rangle. \nonumber
     \end{eqnarray}
   \end{itemize}

\begin{itemize}
     \item Since $\vec{v_0}, \vec{v_1}$ are just lists of distinct names with
      $|\vec{v_0}| = |\meaningof{K_{2}}| = |\vec{v_1}|$, just pick $\vec{v_0} = \vec{v_1}$. Dropping the subscript, we conclude 
     \begin{eqnarray}
       \lefteqn{(\nu \; x_0 \; x_1 \; w_0 \; w_1)\meaningof{K_{1}}\langle \vec{v}:x_0 : x_1 : w_0 : w_1 \rangle \simeq} \nonumber \\
       & & (\nu \; y_{00} \ldots y_{11} \; w_{00} \ldots w_{11})\meaningof{K_{3}}\langle \vec{v_1}:y_{00}:\ldots : y_{11} : w_{00} : \ldots : w_{11} \rangle \nonumber
     \end{eqnarray}
     \item Moreover, we have  $\meaningof{K_{2}}\langle \vec{v} \rangle$ forming the shared core.
   % \item hence $x_0,x_1$ do not occur free in $\vec{v}:y_{00} : y_{01} : y_{10} : y_{11}$,
%      and $y_{00}, y_{01}, y_{10}, y_{11}$ do not occur free in $\vec{v}:x_0 : x_1$
%    \item applying scope extrusion, twice, therefore yields the desired result.
   \end{itemize}

The attentive reader might wonder how R3 works.  Indeed, Perko pairs provide a key counter example to using restriction. This is where bisimulation really shows it's value -- in the form of its flexibility. The outline follows the proof technique of bisimulation-up-to (in this case, up to context), found in \cite{Sangiorgi95} after noting that that there are two simple paths in the knot encoding between any given pair of distinct ports.

\subsection{Reverse direction. $ K_1 \sim K_2  \Leftarrow  \meaningof{K_1} \simeq \meaningof{K_2}$}

We prove this by contradiction,
assuming two knots in distinct isotopy classes with bisimilar images
and arriving at absurdity. Not surprisingly, the argument in this
direction is considerably less complicated as it is nonconstructive.

Without loss of generality (by application of the lemmas of the forward direction as needed), we assume knots are given in minimal
crossing diagrams. Due to this minimality, if the crossing
numbers of these diagrams are different, then the number of free names (or arity of the
abstractions interpreting the knots)  in  $\meaningof{K_1}$  differs from those  in $\meaningof{K_2}$ --contradicting
bisimilarity. Thus, the crossing numbers must be the same. 

Now we employ the form of the encoding: $\Pi_{i=0}^{n-1}
\meaningof{C(i)}(...) | \Pi_{i=0}^{n-1} W(...)|W(...) \simeq
\Pi_{j=0}^{n-1} \meaningof{C(j)}(...) | \Pi_{j=0}^{n-1}
W'(...)|W'(...)$. Notice that a consequence of sharing the same
crossing number is that the crossing process part of the two encodings
is \emph{identical}. Because bisimulation is a congruence this implies
$\Pi_{i=0}^{n-1} W(...)|W(...) \simeq \Pi_{j=0}^{n-1}
W'(...)|W'(...)$. This says that the only difference can be in the
``wiring harnesses''. Obviously, if any of these wires differ (up to $\alpha$-equivalence), then there
is a  distinguishing barb, contradicting our assumption of
bisimulation.  Hence none of the wires differ, their respective sets of crossings are wired
identically, which means the diagrams are identical. Thus, the knots are ambiently isotopic -- a contradiction.

%\subsection{Strengthening the statement of the theorem}
%
%Let
%  
%   \begin{eqnarray}
%     \lefteqn{L_C(\meaningof{K}\langle \vec{u} \rangle, \vec{v}) :=} \nonumber \\
%     & & \{ (\nu \; u)C(\vec{z}) : \exists P \; \meaningof{K}\langle \vec{u} \rangle = (\nu \; u)C(\vec{z}) | P \} \nonumber \\
%     \lefteqn{L_W(\meaningof{K}\langle \vec{u} \rangle, \vec{v}, \vec{w}) :=} \nonumber \\
%     & & \{ W(a,b) : \exists P \; \meaningof{K}\langle \vec{u} \rangle = W(a,b) | P, a,b \in \vec{v}, a, b \not\in \vec{w} \} \nonumber \\
%     \lefteqn{L(\meaningof{K}\langle \vec{u} \rangle, \vec{v}, \vec{w}) :=} \nonumber \\
%     & & \Pi_{C \in L_{c}(\meaningof{K}\langle \vec{u} \rangle, \vec{v})}C | \Pi_{W \in L_{W}(\meaningof{K}\langle \vec{u} \rangle, \vec{v}, \vec{w})} W \nonumber
%   \end{eqnarray}
%
%   we also have
%
%   \begin{eqnarray}
%     L(\meaningof{K_{1}}\langle \vec{v}:\vec{w_1} \rangle,\vec{v},\vec{w_1}:\vec{w_2}) = L(\meaningof{K_{2}}\langle \vec{v}:\vec{w_2} \rangle,\vec{v},\vec{w_1}:\vec{w_2}) \nonumber
%   \end{eqnarray}
%
%  \begin{itemize}
%  \item When the knots are ambient isotopic the encodings
%    \textit{share} a set of crossings and wires at least as big as a
%    minimal crossing representative of the isotopy class.
%  \item And the other parts are R-move complications of wires that
%    would complete the knot from shared core -- hidden under restriction.
%  \end{itemize}

% section proof sketch (end)

\section{Unlikely characters: spatial logic for
  knots}\label{sub:characteristic_formulae} % (fold)

Associated to the mobile process calculi are a family of logics known
as the Hennessy-Milner logics. These logics typically enjoy a
semantics interpreting formulae as sets of processes that when
factored through the encoding outlined above allows an identification
of classes of knots with logical formulae. In the context of this
encoding the sub-family known as the spatial logics \cite{CairesC03}
\cite{CairesC04} \cite{Caires04} are of particular interest providing
several important features for expressing and reasoning about
properties (i.e. classes) of knots. We hint here at how this may be done.

%\begin{description}
%\item [structural connectives] 
\subsubsection{Structural connectives} The spatial logics enjoy
structural connectives corresponding, at the logical level, to the
parallel composition ($P | Q$) and new name ($(\nu \; x)P$)
connectives for processes. As illustrated in the examples below, these
connectives are extremely expressive given the shape of our encoding.
%\item [decideable satisfaction]

\subsubsection{Decideable satisfaction}
In \cite{Caires04} the satisfaction relation is shown to be decideable
for a rich class of processes. It further turns out that the image of
the our encoding is a proper subset of that class. This result
provides the basis for an algorithm by which to search for knots
enjoying a given property.
%\item [characteristic formulae]

\subsubsection{Characteristic formulae}
In the same paper \cite{Caires04} , Caires presents a means of calculating
characteristic formulae, selecting equivalence classes of processes
up to a pre--specified depth limit on the support set of names. Composed with our
encoding, this characteristic formula can be used to select
characteristic formulae for knots.
%\end{description}

\subsubsection{Spatial logic formulae}

The grammar below (segmented for comprehension) summarizes the syntax
of spatial logic formulae. We employ illustrative examples in the
sequel to provide an intuitive understanding of their meaning
referring the reader to \cite{Caires04} for a more detailed explication
of the semantics.

\begin{mathpar}
  \inferrule* [lab=boolean] {} {{A,B} \bc T \;|\; \neg A \;|\; A \wedge B \;|\; \eta = \eta'}
  \and
  \inferrule* [lab=spatial] {} {|\; \pzero \;|\; A | B \;|\; x \text{\textregistered} A \;|\; \forall x . A \;|\;  H x . A}
  \and
  \inferrule* [lab=behavioral] {} {|\; \alpha . A}
  \and 
  \inferrule* [lab=recursion] {} {|\; X(\vec{u}) \;|\; \mu X(\vec{u}) . A}
  \and
  \inferrule* [lab=action] {} {\alpha \bc \langle x?(\vec{y}) \rangle \;|\; \langle x!(\vec{y}) \rangle \;|\; \langle \tau \rangle}
  \and 
  \inferrule* [lab=name] {} {\eta \bc x \;|\; \tau}
\end{mathpar} 

% subsection characteristic_formulae (end)   	 

\subsection{Example formulae}\label{sub:example_formulae_} % (fold)

\subsubsection{Crossing as formula.}
% 
% \begin{align*}
%   \frac{d}{dx} \sin x &= \cos x 
%   & \frac{d}{dx} e^x &= e^x \\
%   \frac{d}{dx} \cos x &= - \sin x 
%   & \frac{d}{dx} \log x &= \frac{1}{x} \\
% \end{align*} 

\begin{align*}
 \mu C(x_{0},x_{1},y_{0},y_{1},u).&(\langle x_{0}?(z) \rangle(\langle u! \rangle\langle y_{1}!z \rangle C(x_{0},x_{1},y_{0},y_{1},u)) & \\
  & \wedge \langle y_{1}?(z) \rangle (\langle u! \rangle \langle x_{0}!z \rangle C(x_{0},x_{1},y_{0},y_{1},u)) & \\
  & \wedge \langle x_{1}?(z) \rangle (\langle u? \rangle \langle y_{0}!z \rangle C(x_{0},x_{1},y_{0},y_{1},u)) & \\
  & \wedge \langle y_{0}?(z) \rangle (\langle u? \rangle \langle x_{1}!z \rangle C(x_{0},x_{1},y_{0},y_{1},u))) &
\end{align*}

The lexicographical similarity between the shape of this formulae and
the shape of definition of the process representing a crossing reveals
the intuitive meaning of this formulae. It describes the capabilities
of a process that has the right to represent a crossing. For example
it picks out processes that may perform an input on the port $x_0$ in
its initial menu of capabilities. What differentiates the formula
from the process, however, is that the crossing process is the
smallest candidate to satisfy the formula. Infinitely many other
processes -- with internal behavior hidden behind this interface, so
to speak -- also satisfy this formula. Even this simple formula,
then, can be seen to open a new view onto knots, providing a
computational interpretation of \emph{virtual} knots.

Note that this formula is derived by hand. A similar formula can be
derived by employing Caires' calculation of characteristic formula
\cite{Caires04} to the process representing a crossing. In light of
this discussion, we let
$\meaningof{C}_{\phi}(x0,x1,y0,y1,u)$ denote a formula specifying the
dynamics we wish to capture of a crossing. To guarantee we preserve
the shape of the interface and minimal semantics we demand that
$\meaningof{C}_{\phi}(x0,x1,y0,y1,u) \Rightarrow
\textbf{C}(x0,x1,y0,y1,u)$ where $\textbf{C}(x0,x1,y0,y1,u)$ denotes
the formula above.
                            
\subsubsection{Crossing number constraints.}
The moral content of the context lemma (Lemma \ref{context}) is that the notion of
``locality'' in the Reidemeister moves is effectively captured by the
parallel composition operator of the process calculus. This intuition
extends through the logic. Given a formula,
$\meaningof{C}_{\phi}(x0,x1,y0,y1,u)$, we can use the structural
connectives to specify constraints on crossing numbers, such as at
least $n$ crossings, or exactly $n$ crossings.
\begin{mathpar}
  \inferrule* [lab=at-least-n] {} { K^{\geq n}_{\phi}(\vec{xs},\vec{ys}) := \Pi_{i=0}^{n-1} Hu . \meaningof{C}_{\phi}(xs_i,ys_i,u) | T }
  \and 
  \inferrule* [lab=exactly-n] {} { K^{= n}_{\phi}(\vec{xs},\vec{ys}) := \Pi_{i=0}^{n-1} Hu . \meaningof{C}_{\phi}(xs_i,ys_i,u) | \neg (\forall x_0,y_0,x_1,y_1,u . \meaningof{C}_{\phi}(x_0,y_0,x_1,y_1,u) | T) }
\end{mathpar}

To round out this section, recall that the encoding of an $n$-crossing
knot decomposes into a parallel composition of $n$ \emph{copies} of a
crossing process together with a wiring harness. To specify different
knot classes with the same crossing number amounts to specifying
logical constraints on the wiring harness. In the interest of space,
we defer examples to a forthcoming paper. Suffice it to say that both
the conditions ``alternating knot'' and ``contains the tangle
corresponding to 5/3'' are expressible. For example, it is possible to
calculate the characteristic formula of a process corresponding to the
tangle 5/3 and conjoin it into the classifying formula via the
composition connective of the logic.

Finally, we wish to observe that it is entirely within reason to
contemplate a more domain-specific version of spatial logic tailored
to the shape of processes in the image of the encoding. Such a
domain-specific logic would have a better claim to the title formal
language of knot properties.

% subsection example_formulae_ (end)

% section knots_as_processes (end) 

% section spatial logic via knots (end)

\section{Conclusions and future work}

\subsection{Knotation system properties}\label{sub:knot_system_props_recap} % (fold)
Rather than list the properties of section
\ref{sub:desirable_properties_for_a_knot_notation_system_} and tick
them off, as we may do, let us take a step back. In the end it is the
act of calculating over knots -- tabulating, searching, making new
ones from old -- that places demands on the notation system used to
represent them. These are the practical considerations that give rise
to the properties listed in
\ref{sub:desirable_properties_for_a_knot_notation_system_}. Establishing
the tight correspondence between processes in the image of the
encoding and the knots they represent is ultimately in service of
showing this is a reasonable proxy for knots because calculations on
the representation have a correspondence to calculations on the
domain. Moreover, the last few decades of logic in computer science
has taught us that there is fidelity and then there is
\emph{fidelity}. It's one thing map elements of one domain to elements
of another, but it's another to map the procedures of one domain
(roughly, operations on elements) to the procedures of
another. Whether one sees a demand like this through the lens of
``functoriality'' or through a correspondence of proofs (of say,
ambient isotopy, in a proof system made of Reidmeister moves) and
programs (that transform programs), the practical side of the demand
is about making calculation \emph{tractable}, so that ``local''
operations in one domain correspond to ``local'' operations in the
other. This makes it possible to structure calculations in terms of
the structure of the elements over which the calculations are being
performed, which has far-ranging consequences on the complexity of
calculation in the representative domain. Of note, we have been
careful to restrict attention to the image of the encoding. We do not,
as of yet, have an effective characterization of processes that will
decide whether or not a process is in the image of the encoding. This
restriction notwithstanding, the image is the only system thus far
proposed that does enjoy all of the properties listed, including possessing (actually, being) a
language for classifying knots via logical properties.

\subsection{Knot queries}\label{sub:knot_queries} % (fold)
The interest in knot tabulation coming from the physical sciences is a
strong motivation to investigate the design of a specialized form of
spatial logic tailored to reasoning about processes in the image of
this encoding. Specifically, the current authors are in the process of
devising an executable knot query language that runs on top of a
translation of spatial logic predicates to XQuery, the XML query
language \cite{W3C-XQ-2873221} \cite{W3C-XQ-5631451}. In this
application, a member, $K$, of an isotopy class $[K]_{\sim}$ is stored
as an XML document $\meaningof{\meaningof{K}}_{XML}$. As the notation
indicates, the document may be calculated from the process expression,
$\meaningof{K}$. Pleasantly, all such documents will conform to an XML
schema \cite{Fallside:04:XSP} \cite{Mendelsohn:04:XSP}
\cite{Biron:04:XSP} that may be formally derived from the grammar for
processes described in section \ref{subsub:process_grammar}. Thus, the
mathematics extends to a formal specification of the software
implementation of a ``knot database'' and the compositional nature of
the specification makes tractable formal verification of the
correctness of the software with regards to the specification.
% subsection knot_queries (end)

\subsection{Braids, tangles, virtual knots, racks and quandles}\label{sub:braids_tangles_and_virtual_knots} % (fold)

As may be seen from the encoding of the Reidemeister moves, nothing in
this approach restricts it to knots. In particular, the same
techniques may be lifted to braids and tangles. More generally,
Kauffman posits an intriguing new member to the knot family by
virtualizing the crossings in a knot diagram
\cite{kauffman-2005-VKNL}. We note that while we arrived at this
encoding before becoming aware of Kauffman's work, that line of
investigation is very much aligned to the guiding intuitions of the
encoding presented here. Namely, the crossing circuit \emph{a priori}
could be any $\pi$-calculus process that respects the interface of the
wiring circuit. We conjecture that many forms of \emph{virtualization}
may be faithfully interpreted as \emph{simulation}.

Moreover, it seems likely that the rack and quandle frameworks may be
mirrored in the process calculus, as these structures can be used to
model the \emph{action} of the Reidemeister moves on a knot. The
relation between this type of action and the dynamics of the process
calculus presented here remains to be investigated.

% subsection braids_tangles_and_virtual_knots (end) 

\subsection{Other calculi, other bisimulations and geometry as behavior}\label{sub:other_calculi_other_bisimulations_and_geometry_as_behavior} % (fold)

The astute reader may have noticed that the encoding
described here is not much more than a linear notation for a minimal
graph-based representation of a knot diagram. In this sense, there appears to be
nothing particularly remarkable about the representation. Though
expressed in a seemingly idiosyncratic way, the encoding is built from the same
information that any freely downloadable program for calculating knot
polynomials uses routinely. What is remarkable about this
representational framework is that it enjoys an \emph{independent}
interpretation as the description of the behavior of concurrently
executing processes. Moreover, the notion of the equivalence of the
behavior of two processes (in the image of the encoding) coincides
exactly with the notion of knot equivalence. It is the precise
alignment of independently discovered notions that often indicates a
phenomena worth investigating.

This line of thought seems particularly strengthened when we recall
the $\pi$-calculus is just one of many `computational calculi' ---the
$\lambda$-calculus being another paradigmatic example--- that may be
thought of as a \emph{computational dynamics+algebra} and that
virtually every such calculus is susceptible to a wide range of
bisimulation and bisimulation up-to techniques
\cite{milner92techniques}. As such, we see the invariant discussed
here as one of many potential such invariants drawn from these
relatively new algebraic structures. It is in this sense that we see
it as a new kind of invariant and is the inspiration for the other
half of the title of this paper.

Finally, if the reader will permit a brief moment of philosophical
reflection, we  conclude by observing that such a connection fits
into a wider historical context. There is a long-standing enquiry and
debate into the nature of physical space. Using the now familiar
signposts, Newton's physics ---which sees space as an absolute
framework--- and Einstein's ---which sees it as arising from and
shaping interaction--- we see this connection as fitting squarely
within the Einsteinian \emph{weltanschauung}. On the other hand,  unlike the particular mathematical framework of continuity in
which Einstein worked out his programme, behavior and its implied notions of space and time are
entirely discrete in this setting, built out of names and acts of
communication. In this light we look forward to revisiting the now
well-established connection between the various knot invariants such
as the Kauffman bracket and quantum groups. Specifically, it appears
that the kinematic picture of loop quantum gravity derived from spin
networks can be faithfully encoded in a manner analogous to one used
here to encode knots, but the process structure offers an account of
dynamics somewhat different from spin foams \cite{baez-2000-543}.

% subsection other_calculi_other_bisimulations_and_geometry_as_behavior (end)

% section conclusion (end)

\section{Appendix: From DT-codes to processes}

What follows is an algorithm for calculating the $\omega$ indexing
function. Note that the function \emph{over} can be implemented easily using the function $\delta$, or the function $sgn$, described in section \ref{ssub:dowker_thistlethwaite_codes}. Exercise for the reader: upgrade the algorithm so that wires
are not duplicated during the process.  \newpage

%  \begin{lstlisting}
\begin{verbatim}
(* ------------------------------------------------------------------ *)
(*                                                                    *)
(* Given types Knot and Wire, the omega function is typed as follows  *)
(* omega : int -> (int -> int) -> (int -> int) -> Knot -> (Wire list) *)
(* With dt a table representing the Dowker-Thistlethwaite mapping     *)
(* from odds to evens, and dti the inverse. The function keeps an     *)
(* accumulator, acc, in which to collect the wires it calculates.     *)
(* The function assumes helper functions C and x0,x1,y0,y1. C takes   *)
(* and instance of type Knot and the odd index of the DT code and     *)
(* returns the crossing. The other helpers are accessors of the ports *)
(* of the crossing. The algorithm visits every crossing and so        *)
(* generates many wires more than once, which is why the accumulator  *)
(* is updated with union rather than cons on the recursive calls. As  *)
(* an exercise, modify the algorithm to avoid wire duplication.       *)
(*                                                                    *)
(* ------------------------------------------------------------------ *)

let omega i dt dti knot acc = 
  if (i <= (numCrossings knot)) 
     then 
       let ic = (2*i - 1) in 
         (omega 
            (i+1) dt dti knot 
            (union acc 
               [ W(x1(C(knot,ic)), 
                 (if (over dt ic-1)
                     then y0(C(knot,ic-1))
                     else y1(C(knot,ic-1)))); 
                 W(y0(C(knot,ic)), 
                 (if (over dt ic+1)
                     then x1(C(knot,ic+1))
                     else x0(C(knot,ic+1)))); 
                 W(x0(C(knot,ic)), 
                 (if (over dt (dti ((dt ic)-1)))
                     then y0(C(knot,(dti ((dt ic)-1))))
                     else y1(C(knot,(dti ((dt ic)-1)))))); 
                 W(y1(C(knot,ic)), 
                 (if (over dt (dti ((dt ic)+1)))
                     then x1(C(knot,(dti ((dt ic)+1))))
                     else x0(C(knot,(dti ((dt ic)+1)))))) ])) 
     else acc
\end{verbatim}
%     \end{lstlisting}

%%% Local Variables:
%%% Tex Master:"jktir.tex"
%%% End: 

% section wiring algorithm (end)

\paragraph{Acknowledgments.}
The authors wish to thank Matthias Radestock for suggestions for
coding wires, and early advocation of the wire-saturated form the
final encoding took and his probing questions regarding proof
strategies in earlier drafts of this paper. One of the co-author
wishes to acknowledge his longstanding debt to Samson Abramksy for
making so accessible his foundational insights into the Curry-Howard
isomorphism. Additionally, the authors would like to thank the
reviewers of ICALP 2006 for their helpful comments on an earlier write
up of these results.

% section acknowledgments (end)

\newpage

\bibliographystyle{plain}   
\bibliography{jktir.bib}

 \end{document}